\documentclass{amsart}
\usepackage{graphicx}
\usepackage{amsmath, amssymb, latexsym}
\title{Edge growth in graph squares}
\author{Michael Goff}

\newtheorem{theorem}{Theorem}[section]

\newtheorem{lemma}[theorem]{Lemma}
\newtheorem{definition}[theorem]{Definition}
\newtheorem{conjecture}[theorem]{Conjecture}

\begin{document}

\begin{abstract}
We resolve a conjecture of Hegarty regarding the number of edges in the square of a regular graph.  If $G$ is a connected $d$-regular graph with $n$ vertices, the graph square of $G$ is not complete, and $G$ is not a member of two narrow families of graphs, then the square of $G$ has at least $(2-o_d(1))n$ more edges than $G$.
\end{abstract}

\date{October 3, 2011}

\maketitle

\section{Introduction}
In this paper, we consider the following problem.  Let $G$ be a $d$-regular graph, and let $G^2$ be the graph with the same vertex set as $G$ and an edge $(u,v)$ whenever $u$ and $v$ are within distance $2$ in $G$.  Then find a lower bound on the number of edges of $G^2$, or $e(G^2)$.  With the assumptions that $G$ is connected and that $G^2$ is not a complete graph, this question was posed by Hegarty \cite[Conjecture 1.8]{CD}.

In his work, Hegarty discussed general graph powers.  Let $G^k$ be the graph with an edge $(u,v)$ whenever $u$ and $v$ are within distance $k$ in $G$.  Several authors have considered lower bounds on $e(G^k)$.  Hegarty found that $e(G^3) \geq (1+c)dn/2$ if $G$ is a $d$-regular graph with diameter at least three, with $c = 0.087$.  Pokrovskiy \cite{graphpowers} found a value of $c = 1/6$, and DeVos and Thomass\'{e} improved the value of $c$ to $3/4$ and provided examples to demonstrate that no higher value of $c$ is possible.  The latter authors also weakened the $d$-regular condition to a minimum degree of $d$.

DeVos, McDonald, and Scheide \cite{avgdeg} considered higher powers of $G$.  They found that if $G$ has a minimum degree of $d \geq2$ and $G$ has at least $\frac{8}{3} d$ vertices, then $G^4$ has an average degree of at least $\frac{7}{3} d$.  Examples demonstrate that neither the $8/3$ nor the $7/3$ constants may be improved.  They also found that when the diameter is at least $3k+3$, then the average degree of $G^{3k+2}$ is at least $(2k+1)(d+1) - k(k+1)(d+1)^2/n - 1$, and examples show that this cannot be improved.

Inspiration for Hegarty's work comes from the Cauchy-Davenport theorem, which states that if $A$ is a subset of $\mathbb{Z}_p$ for a prime $p$, and $hA$ denotes the set of sums of collections of $h$ elements of $A$, then $|kA| \geq \min(p, k|A| - (k-1))$.  Knesser \cite{knesser} generalized the Cauchy-Davenport theorem to an abelian group $H$.  Now suppose that $A = -A$ and that $A$ contains the identity element.  The connection to graph theory comes through the Cayley graph.  The Cayley graph $G(H,A)$ has vertex set $H$ and an edge $g_1g_2$ whenever $g_1 - g_2 \in A$.  Then $(G(H,A))^k = G(kH, A)$, and the growth in $kH$ is equivalent to the growth in the vertex degrees of $(G(H,A))^k$.

For large values of $d$ and $k > 2$, $e(G^k)$ exceeds $e(G)$ by at least a constant factor, so long as $G^k$ is not a complete graph and  $G$ is $d$-regular.  This is not true for $k=2$, as examples demonstrate.  In that case, Hegarty conjectures the following \cite[Conjecture 1.8]{CD}.

\begin{conjecture}
\label{square_conj}
Let $G$ be a connected $d$-regular graph with $n$ vertices such that $G^2$ is not a complete graph.  Then $e(G^2) - e(G) \geq (2-o_d(1))n$.
\end{conjecture}

This conjecture is not true as stated, and we shall see some counterexamples below, but the counterexamples are confined to narrow families of graphs known as snake graphs and peanut graphs.  Conjecture \ref{square_conj} is true if we claim that $e(G^2) - e(G) \geq (3/2-o_d(1))n$.  However, our main theorem is as follows.

\begin{theorem}
\label{main_theorem}
Let $G$ be a connected $d$-regular graph with $n$ vertices such that $G^2$ is not a complete graph and $d > 6$.  Also suppose that $G$ is not a snake graph or a peanut graph.  Then $$e(G^2) - e(G) > 2n\left( 1 -\frac{2}{d+1} - \frac{3}{d-3} \right).$$
\end{theorem}

Our approach is as follows.  We define basic terms in Section \ref{def_sect}.  We rephase the problem by counting ordered pairs of vertices $(u,v)$ such that $u$ and $v$ are distance $2$.  In Section \ref{region_sect}, we divide $G$ into what we call regions, and into superregions in Section \ref{super_sect}, and for each superregion $R$, we associate a collection of pairs $S_R$ such that $|S_R| > 4|R| \left( 1 -\frac{2}{d+1} - \frac{3}{d-3} \right)$.  A particularly important type of superregion is the class of tails, which we discuss in Section \ref{tail_sect}.  It is necessary to show that the superregions are a partition of the vertices of $G$, which we do in Section \ref{part_sect}.  In Section \ref{exception_sect}, we discuss the snake graph and peanut graph in detail.  We complete the proof in Section \ref{proof_sect} by showing that the collections of pairs associated with distinct superregions are also disjoint.

\section{Definitions}
\label{def_sect}
This section contains basic definitions that are used for the rest of the paper.

Let $G$ be a graph without multiple edges or loops.  $V(G)$ denotes the vertex set of $G$, and $e(G)$ the edge set of $G$.  If $X$ is a set of vertices of $G$, then $G[X]$ is the induced subgraph on $X$, or the maximal subgraph of $G$ with vertex set $X$.

The \textit{distance} between two vertices $u$ and $v$, denoted by $d(u,v)$, is the number of edges in a shortest path between $u$ and $v$.  Thus $d(u,u) = 0$, $d(u,v) = 1$ if there is an edge $uv$, and so on.  For each $i \geq 1 $ and vertex $v$, let $N_i(v)$ be the set of vertices that are distance $i$ from $v$.  We also say that $N(v) := N_1(v)$, and $\deg_i(v) := |N_i(v)|$.  Also, $N'_2(v)$ is the set of vertices $u \in N_2(v)$ such that $u \in N(w)$ for some $w \in N_3(v)$.  A $d$-regular graph is a graph such that every vertex $v$ satisfies $\deg(v) := \deg_1(v) = d$.

Let the graph power $G^k$ be the graph with $V(G^k) = V(G)$ and an edge $uv$ whenever $d(u,v) \leq k$ in $G$.

A \textit{low degree vertex} $v$ is a vertex $v$ satisfying $\deg_2(v) \leq 3$.  Note that if $v$ is a low degree vertex and $N'_2(v) = \emptyset$, then $G$ contains at most $d+4$ vertices and thus $G^2$ is complete when $d > 6$.
\begin{lemma}
\label{HighDegLemma}
Let $v \in V(G)$ and let $u \in N_2'(v)$.  Then $\deg_2(u) \geq d-\deg_2(v) + 1$.  In particular, if $v$ is low degree, then $\deg_2(u) \geq d-2$.
\end{lemma}
\begin{proof}
$G-N_2'(v)$ is disconnected, with one component of $G-N_2'(v)$ being $G_v := \{v\} \cup N(v) \cup (N_2(v) - N_2'(v))$ and another component $G'$ containing a vertex $x' \in N(u)$ that is distance $3$ from $v$.  Since $|N_2'(v)| \leq \deg_2(v)$, $x'$ has degree at least $d - \deg_2(v)$ in $G'$, and let $N'(x')$ be the sets of neighbors of $x'$ in $G'$.  Also, choose $x \in G_v \cap N(u)$.  Since $N(x) \subset \{v\} \cup N(v) \cup N_2(v)$, the sets $N(x) \cup \{x\}$ and $N'(x') \cup \{x'\}$ are disjoint, and thus $N(x) \cup N'(x') \cup \{x, x'\} \geq 2d-\deg_2(v) + 2$.  Every vertex of $N(x) \cup N'(x') \cup \{x, x'\}$ is within distance $2$ of $u$, and since at most $d+1$ of them are within distance $1$ of $u$, we have that $N(x) \cup N'(x') - N_2(u) \geq d-\deg_2(v) + 1$.  This proves the lemma.
\end{proof}

For the remainder of this paper, we assume that $G$ is a $d$-regular graph such that $G^2$ is not complete.

\section{Regions}
\label{region_sect}

In this section, we discuss regions, a key tool in the proof of our main theorem.  Let $X$ be the set of vertices $x \in G$ satisfying $\deg_2(x) < 4$.  Define an equivalence relation on $X$ by saying that $u \sim v$ if there exists a sequence of vertices $(u=v_0, v_1, \ldots, v_t = v)$ such that for $0 \leq i \leq t-1$, $v_i \in X$ and $d(v_i, v_{i+1}) \leq 2$.

\begin{definition}
A \textit{region} is an equivalence class $X' \subset X$ under $\sim$ together with all neighbors of vertices in $X'$.
\end{definition}
Note that some vertices might not be contained in any region.  We prove some basic properties of regions.

\begin{lemma}
\label{generalregionlemma}
Let $X$ be the set of vertices $x$ satisfying $\deg_2(x) < 4$, and let  $R$ be a region containing $v \in X$.  Then \\
1) $R$ contains at least $d+1$ vertices. \\
2) $R \subset \{v\} \cup N(v) \cup N_2(v)$. \\
3) Let $t := \min \{ \deg_2(x) | x \in R\}$.  Then $R$ contains at most $d+t+1$ vertices. \\
4) $R$ contains at most $d+4$ vertices. \\
5) $R$ is disjoint from any other region $R'$.
\end{lemma}
\begin{proof}
Part 1 follows from the fact that, by definition, $R$ contains $v$ and all neighbors of $v$.

Note that $G-N_2'(v)$ is disconnected, and the component of $G-N_2'(v)$ that contains $v$ is $G_v := \{v\} \cup N(v) \cup (N_2(v) - N_2'(v))$.  Let $G'$ be a different component. By Lemma \ref{HighDegLemma} and $d>6$, no vertex in $N_2'(v)$ is in $X$.  Consider $u \in N_2'(v)$.  If some $w \in N(u) \cap G_v$ is in $X$, then no $w' \in N(u) \cap G'$ is also in $X$.  To see this, observe that $N_2(w)$ contains $N(u)-G_v$, and so $|N(u) - G_v| < 4$ and $|N(u) \cap G_v| > d-4$.  Since $N_2(w')$ contains $N(u) \cap G_v$, $\deg_2(w') > d-4 = 3$ and $w' \not \in X$.

Suppose by way of contradiction that there exists $v' \in R \cap X$ outside of $G_v$.  By definition of a region, there is a sequence of vertices $\{v = v_0, v_1, \ldots, v_k = v'\}$ such that each $v_i \in X$ and $d(v_i, v_{i+1}) \leq 2$ for all $0 \leq i \leq k-1$.  Let $v_j$ be the first vertex in the sequence outside of $G_v$.  Since $G-N_2'(v)$ is disconnected and no vertex in $N_2'(v)$ is also in $X$, $d(v_{j-1}, v_{j}) = 2$, and there exists a vertex $u \in N_2'(v)$ adjacent to both $v_{j-1}$ and $v_j$.  This contradicts the previous paragraph.  Part 2 follows.  There exists $v \in R \cap X$ with $\deg_2(v) = t$, and Part 3 follows.  Part 4 follows by $t < 4$.

Now consider $x \in R \cap R'$.  By definition of a region, there exist vertices $v \in R \cap X$ and $v' \in R' \cap X$ such that $d(v,x) \leq 1$ and $d(v',x) \leq 1$.  Then $d(v,v') \leq 2$ and thus $R=R'$.  This proves Part 5.
\end{proof}

We define several classes of regions.  Let $R$ be a region with a vertex $v$ satisfying $\deg_2(v) = 1$, and let $N_2(v) = \{u\}$.  Let $G_v$ the component of $G-u$ that contains $v$.  Then $G_v$ contains $d+1$ vertices, namely $\{v\} \cup N(v)$.  In $G_v$, all neighbors of $u$ have degree $d-1$, and all other vertices have degree $d$.  Hence the complement of $G_v$ is a matching on the neighbors of $u$.  Let $t := |G_v \cap N(u)|$, and consider $w \in G_v \cap N(u)$.  Since $u$ contains $d-t$ neighbors outside of $G_v$, $\deg_2(w) = d-t+1$, and $w$ is low degree if and only if $t \geq d-2$.  Then $R$ is either $G_v$ or $G_v \cup \{u\}$, and the latter holds if and only if $t \geq d-2$.  If $t \neq d-1$, then we say that $R$ is an \textit{A region}, and if $t=d-1$, we say that $R$ is a \textit{B region}.  Since a B region contains the complement of a matching on $d-1$ vertices, a B region exists only when $d$ is odd.

Now let us say that $R$ is not an A or B region, but $R$ contains a vertex $v$ satisfying $\deg_2(v) = 2$.  First suppose that $R$ contains a vertex $v'$ with $\deg_2(v') = 2$ and $|N_2'(v')| = 1$.  Then we say that $R$ is a \textit{C region}.  Otherwise, we say that $R$ is a \textit{D region}.

Now suppose that $R$ is a region such that $\deg_2(v) = 3$ for all low degree vertices $v$ in $R$.  Let $k$ be the minimum size of $N_2'(v)$ for low degree vertices in $R$.  In the respective cases that $k = 1, 2,$ or $3$, we say that $R$ is an \textit{E region}, an \textit{F region}, or a \textit{G region}.

\section{Superregions}
\label{super_sect}

We now define \textit{superregions}.  We show that the superregions of $G$ are a partition of $V(G)$ in Section \ref{part_sect}.  Before we specify the superregions, we first define sets of vertices associated with $G$ and the superregions.

Let $\mathcal{U}$ be the set of vertices $u$ in $G$ that satisfy $\deg_2(u) \geq d-2$.  By Lemma \ref{HighDegLemma}, if $v$ is a low degree vetex, then $N_2'(v) \subset \mathcal{U}$.  In defining superregions, we will also designate special sets $\mathcal{W}$ and $\mathcal{N}$ such that if $R$ is a superregion, then $|R \cap (\mathcal{W} \cup \mathcal{N} - \mathcal{U})| \leq \frac{2}{d+1}|R|$.  Let $\mathcal{V} := V(G) - \mathcal{U} - \mathcal{W} - \mathcal{N}$.

\begin{lemma}
Theorem \ref{main_theorem} holds for $G$ if $|\mathcal{U}| \geq \frac{3}{d-3}n$.
\end{lemma}
\begin{proof}
If $v \not \in \mathcal{U}$, then $\deg_2(v) \geq 1$.  We then have that $\sum_{v \in V(G)} \deg_2(v) \geq \frac{3}{d-3}n (d-2) + \frac{d-6}{d-3}n(1) \geq 4n$ as desired.
\end{proof}

We therefore assume that $$|\mathcal{U}| < |V(G)|\frac{3}{d-3},$$ and then, since superregions partition $V(G)$, $$|\mathcal{V}| = |V(G) - |\mathcal{W}| - |\mathcal{N}| - |\mathcal{U}| > |V(G)|\left( 1 -\frac{2}{d+1} - \frac{3}{d-3} \right).$$

For every superregion $R$, we associate a collection $S_R$ of at least $4|R \cap \mathcal{V}|$ ordered pairs of vertices of the form $(x,y)$ such that $d(x,y) = 2$.  Since the superregions partition $V(G)$, we have that $\sum_R |S_R| \geq 4|\mathcal{V}|$.  This proves Theorem \ref{main_theorem} as long as the $S_R$ are disjoint, a matter that is partially addressed below and addressed more fully in Section \ref{proof_sect}.  We partition $S_R$ into subsets $S_{Ri}$, for $1 \leq i \leq 4$, as follows.  For each pair $(x,y) \in S_R$,
\begin{list}{$\bullet$}{}
\item if $x$ and $y$ are both in $R$, then $(x,y) \in S_{R1}$, 
\item if $x \in \mathcal{V} \cap R$ and $y \not \in R$, then $(x,y) \in S_{R2}$, 
\item if $y \in \mathcal{V} \cap R$ and $x \in \mathcal{U}-R$, then $(x,y) \in S_{R3}$, 
\item if $x \in \mathcal{W} \cap R$ and $y \not \in R$, then $(x,y) \in S_{R4}$.
\end{list}

\begin{lemma}
\label{duppairs}
Suppose that $(x,y) \in S_{R} \cap S_{R'}$ for distinct superregions $R$ and $R'$.  Then either $(x,y) \in S_{R4} \cap S_{R'3}$ or $(x,y) \in S_{R3} \cap S_{R'4}$.
\end{lemma}
\begin{proof}
By definition of the $S_{Ri}$, one of $x$ or $y$ is in $R$ and the other is in $R'$.  Without loss of generality, assume that $x \in R$ and $y \in R'$.  Then $(x,y) \in S_{R'3}$.  Then $x \not \in \mathcal{V}$, and so $(x,y) \in S_{R4}$.
\end{proof}

As we specify the superregions and the sets $S_R$ over the rest of this section and Section \ref{tail_sect}, observe that the following holds by construction.

\begin{lemma}
\label{lemma31}
Let $v$ be a low degree vertex of a superregion $R$.  Then there are at most $4-\deg_2(v)$ vertices $u$ such that $(u,v) \in S_{R3}$.
\end{lemma}

\subsection{Single vertex superregions}

If $v$ is a vertex that is not contained in any region, then say that $\{v\}$ is a superregion $R$.  Since $v$ is not in a region, $\deg_2(v) \geq 4$.  If $v \not \in \mathcal{V}$, then $S_R = \emptyset$.  Otherwise, let $S_R = S_{R2}$ be $\{(v,u): u \in N_2(v)\}$.

\subsection{D, E, F, G regions}

If $R$ is a region that is not contained in any other superregion, then $R$ is a superregion.  If $R$ is a D, E, F, or G region, for all $v \in R \cap \mathcal{V}$, $\deg_2(v) + |N_2'(v)| \geq 4$.  Choose $A_v \subseteq N_2'(v)$ so that $A_v = N_2'(v)$ if $\deg_2(v) = 2$, $|A_v| = 1$ if $\deg_2(v) = 3$, and otherwise $A_v = \emptyset$.  Set $$S_{v} := \{(v,a)\}_{a \in N_2(v)} \cup \{(a,v)\}_{a \in A_v}.$$

Then set $S_{R} := \cup_{v \in R \cap \mathcal{V}} S_{v}$.  The pairs $\{v,a\}$ are either in $S_{R1}$ or $S_{R2}$, and the pairs $\{a,v\}$ are either in $S_{R1}$ or $S_{R3}$.  Since $A_v \cap \mathcal{V} = \emptyset$, the $S_v$ are disjoint, and thus $|S_R| \geq 4|R \cap \mathcal{V}|$.  The sets $A_v$ are not chosen arbitrarily, but the choice will be made strategically in order to insure that $S_R \cap S_{R'} = \emptyset$ for all superregions $R' \neq R$.  This matter is discussed more fully below.

\subsection{A regions}

Next, suppose that $R$ is an A region and a superregion.  Let $V$ be the set of vertices $v \in R$ with $\deg_2(v) = 1$, let $u$ be the unique vertex in $N_2(v)$ for each vertex in $v \in V$, and let $X$ be the set of remaining vertices of $R$.  Choose distinct $w_1, w_2 \in X$ and set $R \cap \mathcal{W} = \{w_1, w_2\}$.  By definition of an A region, $|V| \geq 3$.  Note that $\deg_2(x) = |V|$ for $x \in X$.  Let $S_{R}$ be the the union of the following sets of pairs: 
\begin{list}{$\bullet$}{}
\item if $|V| < d-2$ (so that $X \cap \mathcal{U} = \emptyset$), all $(|X|-2)(|V|)$ pairs of the form $(x,y)$ for $x \in X-\mathcal{W}$ and $y \in N_2(x)$ (these pairs are either in $S_{R1}$ or $S_{R2}$);
\item $(|X|-2)\max(0,4-|V|)$ pairs of the form $(y,x)$ for $x \in X-\mathcal{W}$ and $y \in N_2(x) - R$ (call this set $S'_R$);
\item all $|V|$ pairs $(v,u)$ for $v \in V$ (these pairs are either in $S_{R1}$ or $S_{R2}$);
\item all $|V|$ pairs $(u,v)$ for $v \in V$ (these pairs are either in $S_{R1}$ or $S_{R3}$);
\item and all $2|V|$ pairs of the form $(w,y)$ for $w \in \{w_1, w_2\}$ (these pairs are in $S_{R4}$ if $y \not \in R$ and in $S_{R1}$ otherwise).
\end{list}
Then $|S_{R}| \geq 4(|V| + |X| - 2) = 4|R \cap \mathcal{V}|$.

Observe that $S'_R \subseteq S_{R3}$.  To see this, if $(|X|-2)\max(0,4-|V|) \neq 0$, which happens only if $|V| = 3$, then for $x \in X - \mathcal{W}$, $\deg_2(x) = 3$.  Let $N_2(w) = \{x',y,y'\}$ for $x' \in R$ and $y, y' \not\in R$.  The only possible neighbors of $y$ in $N_1(x) \cup N_2(x)$ are $u$ and $y'$, and thus $y$ is adjacent to some vertex that is distance $3$ from $x$.  The same is true for $y'$.  Thus $\{y, y'\} \subset N_2'(x)$, showing that $(y,x), (y',x) \in S_{R3}$.

\subsection{C regions}

Suppose that $R$ is a C region.  Let $v$ be a vertex in $R$ with $\deg_2(v) = 2$ and $N_2'(v) = \{u\}$.  Let $G_v$ be the component of $G-\{u\}$ that contains $v$.  Let $V$ be the set of vertices that are distance $2$ from $u$ and in $G_v$.  Choose a vertex $w$ in $R \cap N(u)$ and set $R \cap \mathcal{W} = \{w\}$, and let $X$ be the set of remaining vertices in $R \cap N(u)$.  We prove several lemmas on the structure of C regions.

\begin{lemma}
\label{cregionlemma2}
With all quantities as above, $V(R) \subseteq V \cup X\cup \{w,u\}$.
\end{lemma}
\begin{proof}
By Lemma \ref{generalregionlemma}, $V(R) \subset G_v \cup \{u\}$.  Furthermore, $G_v \cup \{u\} = \{v\} \cup N(v) \cup N_2(v)$ and thus $|G_v| = d+2$.  It suffices to show that if $v' \in R$, then $d(v',u) \leq 2$.  Suppose by way of contradiction that $d(v',u) \geq 3$.  Then $\{v'\} \cup N(v') \cup N_2(v') \subseteq G_v$, and thus $\deg_2(v') \leq 1$, a contradiction to the definition of a C region.
\end{proof}

\begin{lemma}
\label{cregionlemma1}
Let all quantities be as above.  Then $G_v = R$.
\end{lemma}
\begin{proof}
First we show that $G_v \subset R$.  Since $G_v$ has $d+2$ vertices and $N(v) \subset G_v$, there exists one vertex $z \in G_v$ that is not adjacent to $v$.  By definition of a region, all vertices of $G_v - \{z\}$ are in $R$.  Each $v' \in V$ is in $R$ since $v'$ is within distance $2$ of $v$ and is low degree.  Since $u$ must have at least one neighbor outside of $G_v$, $u$ has at most $d-1$ neighbors in $G_v$, and by Lemma \ref{cregionlemma2}, $|V| \geq 3$.  Then since $\deg(z) = d$, at least one vertex $v' \in V$ is adjacent to $z$.  Thus $z \in R$ since $v'$ is a low degree vertex in $R$.

Finally, we must show that $u \not \in R$.  Note that $|V| = 3$ is impossible, since then $G_v$ would contain $d-1$ vertices of degree $d-1$ and $3$ vertices of degree $d$, and the sum of the degrees of all vertices would be odd.  Thus $|V| \geq 4$, and since $u$ has $d+2-|V|$ neighbors in $G_v$, $u$ has $|V|-2 \geq 2$ neighbors outside of $G_v$.  If $x \in R \cap N(u)$, then $N_2(x)$ contains $2$ vertices in $G_v$ and $|V|-2 \geq 2$ outside of $G_v$, and thus $x$ is not low degree.  By construction of a region, $u \not \in R$.
\end{proof}

\begin{lemma}
\label{cregionlemma}
If $R$ is a C region, and $R$ contains a vertex $v$ with $\deg_2(v) = 2$ and $N_2'(v) = \{u\}$, then $R$ contains at least two vertices adjacent to $u$.
\end{lemma}
\begin{proof}
Suppose by way of contradiction that $R \cap N(u) = \{x\}$.  By Lemma \ref{cregionlemma1}, $R-\{u\} = \{v\} \cup N(v) \cup N_2(v) - N_2'(v)$.  Thus some vertex $x'$ of $R-\{u\}$ is not adjacent to $x$, and $d(x',u) \geq 3$.  This contradicts Lemma \ref{cregionlemma2}.
\end{proof}

Let $y_1$ and $y_2$ be two distinct neighbors of $u$ outside of $R$, which exist as in the proof of Lemma \ref{cregionlemma1}.  Also observe that $u \in \mathcal{U}$.  Let $S_{R}$ be the union of the following sets:
\begin{list}{$\bullet$}{}
\item the $|V| + 2|X|$ pairs of the form $(x,y)$ for $x \in V \cup X$ and $y \in R$.  These pairs are each in $S_{R1}$, and they exist since every vertex in $V$ and $X$ have, respectively $1$ and $2$ non-neighbors in $G_v$.  Furthermore, $(G_v)^2$ is complete since in $G_v$, every vertex has degree at least $d-1$ and there are $d+2$ vertices total;
\item the $|V|$ pairs $(v',u)$ for $v' \in V$ (these pairs are in $S_{R2}$);
\item the $|V|$ pairs $(u,v)$ for $v \in V$ (these pairs are in $S_{R3}$);
\item the $|V|$ pairs of the form $(w,y)$ for $y \in N_2(w)$ (these pairs are either in $S_{R1}$ or $S_{R4}$);
\item and if $|V| < d-2$, the $2|X|$ pairs $(x,y_1), (x,y_2)$ for each $x \in X$ (call this set of pairs $S'_R$).
\end{list}
As in the proof of Lemma \ref{cregionlemma1}, $\deg_2(x) = |V|$ for all $x \in R \cap N(u)$, and the condition that $|V| < d-2$ is equivalent to $x \not \in \mathcal{U}$ for any $x \in R \cap N(u)$, and so $S'_R \subset S_{R2}$.  We have that $|S_R| \geq 4|R \cap \mathcal{V}|$.

B regions are discussed in the context of tails in the next section.

\section{Tails and superregions}
\label{tail_sect}

In this section we define several types of superregions that are based on a tail subgraph.  Every superregion described in this section contains a B region, and thus these superregions exist only when $d$ is odd.

\subsection{Tails}

The following construction, a tail, is a subgraph of all superregions defined in this section.  A tail itself is a superregion unless it is contained in a larger superregion.

\begin{definition}
A tail is constructed as follows.  Let $R_1$ be a B region, and for $2 \leq i \leq k$, let $R_i$ be a clique on $d+1$ vertices with the edge $v_{i}v_{i}'$ removed.  Let $v$ be the degree $d-1$ vertex of $R_1$.  Add edges $vv_{2}$ and $v_{i}'v_{i+1}$ for $2 \leq i \leq k-1$.  See Figure \ref{tail_figure}.
\end{definition}

The $R_i$ are the \textit{segments} of $T$.  We may have $k=1$, in which case the tail is a B region.  An \textit{improper tail} is contained in a larger tail, and otherwise a tail is \textit{proper}.  If $T$ is a tail, let $u_T$ be the unique vertex that is adjacent to a vertex in $T$ but not itself in $T$, and let $w_T$ be the unique vertex in $T$ that is adjacent to $u_T$.

\begin{figure}[h]	
\centerline{
\mbox{\includegraphics[width=4.50in]{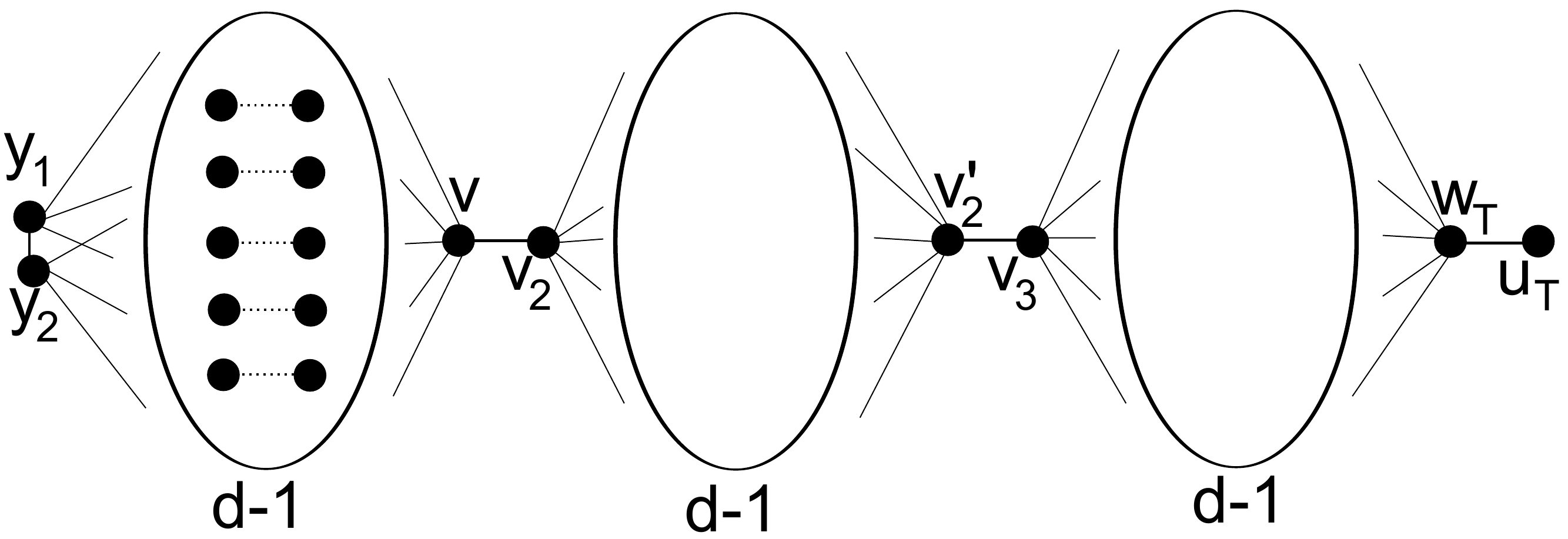}}
}
\caption{Tail.  Note that $u_T$ is not part of the tail.}
\label{tail_figure}
\end{figure}

Let $T$ be a tail.  Let $\mathcal{N} \cap T$ be $\{y_1, y_2\}$ such that $y_1$ and $y_2$ are the two vertices of $R_1$ that satisfy $\deg_2(y_1) = \deg_2(y_2) = 1$.  Associate the following sets of pairs of vertices with $T$: 
\begin{list}{$\bullet$}{}
\item $S_{T1}$ is the set of $(4k-3)(d-1)$ pairs of vertices $(x,y)$ and $(y,x)$ such that $x, y \in T$, $d(x,y) = 2$, and $\deg_2(x) = 2$,
\item $S_{T2}$ is the set of $d-1$ pairs of vertices of the form $(x,u_T)$ such that $x \in T$, $d(x,u_T) = 2$, $\deg_2(x) = 2$,
\item $S_{T3}$ is the set of $d-1$ pairs of vertices of the form $(u_T,x)$ such that $x \in T$, $d(x,u_T) = 2$, $\deg_2(x) = 2$.
\end{list}

In the event that $T$ is a superregion, let $Y$ be the set of $d-1$ neighbors of $u_T$ that are not in $T$ and set $T \cap \mathcal{W} = \{w_T\}$. Then say that
\begin{list}{$\bullet$}{}
\item $S_{T4}$ is the set of $d-1$ pairs $(w_T,y)$ for $y \in Y$.
\end{list}
Then $T \cap \mathcal{V}$ contains $d-1$ vertices in each segment of $T$, and $|S_T| = |S_{T1}| + |S_{T2}| + |S_{T3}| + |S_{T4}| = 4|T \cap \mathcal{V}|$.

We now define a snake graph, our first exception to Conjecture \ref{square_conj}.  See Figure \ref{snake_figure} for an illustration.

\begin{definition}
A \textit{snake graph} $G$ is the union of two tails $T \cup T'$, together with an edge $u_Tu_{T'}$.
\end{definition}

\begin{figure}[h]	
\centerline{
\mbox{\includegraphics[width=4.50in]{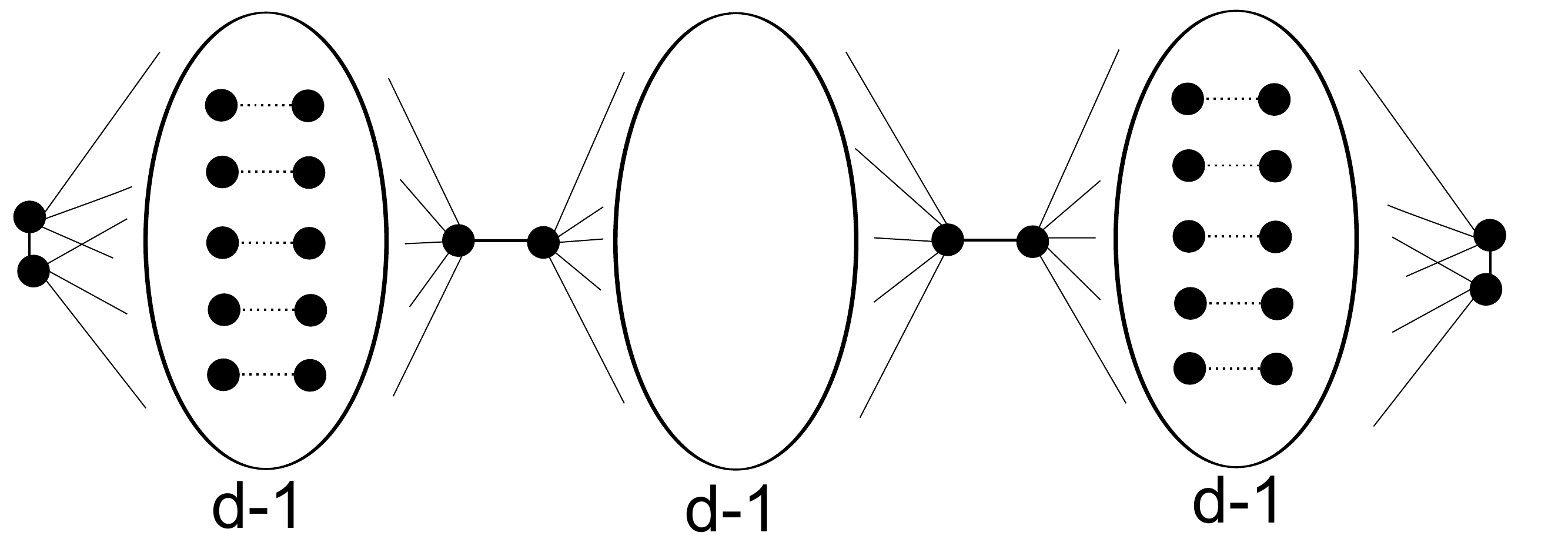}}
}
\caption{Snake graph}
\label{snake_figure}
\end{figure}

The two superregions of $G$ above are $G-R_1$ and $G-R'_1$, where $R_1$ and $R'_1$ are the two B regions.  These superregions intersect if either $T$ or $T'$ contain more than $1$ segment.

We prove an important lemma on tails.

\begin{lemma}
\label{tail_intersect}
Let $T$ and $T'$ be tails with nonempty intersection.  Then either $G$ is a snake graph or one of $T$ and $T'$ is contained in the other.
\end{lemma}
\begin{proof}
Let the segments of $T$ and $T'$ be $R_1, R_2, \ldots, R_k$ and $R'_1, \ldots, R'_j$ respectively, and without loss of generality, $k \geq j$.  Also suppose that $R_1$ and $R_1'$ are B regions, and that $R_i$ and $R_{i+1}$ are joined by an edge for $1 \leq i \leq k-1$, as are $R'_i$ and $R'_{i+1}$ for $1 \leq i \leq j-1$.  Since a tail is a union of regions, $T \cap T'$ is also a union of regions.

First suppose that $R_1 = R_1'$, and we show by induction that $T' \subset T$.  Let $v$ be the vertex of $R_1$ that is adjacent to a vertex $x$ in $R_2$.  Since $v$ is the only vertex of $R_1$ with degree less than $d$ in $R_1$, then $v$ is also adjacent to a vertex $y$ in $R'_2$.  Also, $v$ has only one neighbor outside of $R_1$, and so $x=y$.  Then $R_2$ and $R'_2$ intersect, and thus they must be equal.

Now suppose that $R_{i-1} = R'_{i-1}$, and let $v$ and $v'$ be the two vertices in $R_{i-1}$ that do not have degree $d$ in $R_{i-1}$.  Suppose that $v'$ is the vertex in $R_{i-1}$ that is adjacent to a vertex in $R_{i-2}$.  Then $v$ is adjacent to a vertex $x$ in $R_i$.  One of $v$ and $v'$ is adjacent to a vertex $y$ in $R'_{i}$, and since $v'$ is adjacent to a vertex in $R'_{i-2} = R_{i-2}$, then $v$ is adjacent to $y$.  Since $v$ is adjacent to only one vertex outside of $R_{i-1}$, then $x=y$.  Then since $R_i$ and $R'_i$ intersect, they are equal.  Thus if $R_1 = R'_1$, then $T' \subset T$.

Now consider the case that $R_1 \neq R'_1$.  Choose $k'$ and $j'$ so that $R_{k'} = R'_{j'}$ and the sum $j'+k'$ is minimized.  Let $T_1$ be $R_1 \cup \ldots \cup R_{k'-1}$, and let $T_2$ be $R'_1 \cup \ldots \cup R'_{j'-1}$.  Let $v$ and $v'$ be the two vertices of $R_{k'} = R'_{j'}$ that are adjacent to vertices outside of $R_{k'}$.  Since $T_1$ and $T_2$ are disjoint, $w_{T_1} \neq w_{T_2}$.  Then either $v$ or $v'$ is adjacent to $w_{T_1}$; assume without loss of generality an edge $vw_{T_1}$.  Either $vw_{T_2}$ or $v'w_{T_2}$ is an edge, and since $v$ is adjacent to only one vertex outside of $R_k$, there is an edge $v'w_{T_2}$.  There can be no other vertices or edges in $G$, as the above establishes a $d$-regular graph.  Thus $G$ is a snake graph.
\end{proof}

\subsection{Multitails}

Let $\{T_1, \ldots, T_m\}$ be a maximal collection of tails such that $u_{T_1} = \ldots = u_{T_m}$.  If $m \geq 2$, then $R = T_1 \cup \cdots \cup T_m$ is an $m$-\textit{multitail}, and multitails are superregions.  Let $S_R$, each $S_{Ri}$ for $1 \leq i \leq 4$, $R \cap \mathcal{N}$, and $R \cap \mathcal{W}$ be the union of the corresponding sets over the $T_j$.  Then $|S_R| = 4|R \cap \mathcal{V}|$.

\subsection{A tails}

We define our next superregion $R$, an \textit{A tail}, as follows.  See Figure \ref{atail_figure} for an illustration.

\begin{definition}
Let $T$ be a tail with $k$ segments.  Let $H'$ be a subgraph of $G$ consisting of the following vertices: 
\begin{list}{$\bullet$}{}
\item $u_T$, 
\item a vertex $z$, 
\item a set $X$ with $|X| = d-2$, 
\item a subset $X' \subset X$ of even cardinailty less than $d-3$, and 
\item vertices $y_1$ and $y_2$.
\end{list}
Let the edge set of $H'$ be as follows: 
\begin{list}{$\bullet$}{}
\item $u_Tz$, 
\item $u_Tx$ for all $x \in X$, 
\item a complete subgraph on $X$ with a matching on $X'$ removed, 
\item $xy_1$ and $xy_2$ for all $x \in X$, 
\item $y_1y_2$, 
\item $zy_1$ and $zy_2$, and 
\item $xz$ for all $x \in X'$.
\end{list}
All vertices $x \in X-X'$ are low degree and are all in the same region $H$.  Then we say that $T \cup H$ is an A tail.
\end{definition}

\begin{figure}[h]	
\centerline{
\mbox{\includegraphics[width=4.70in]{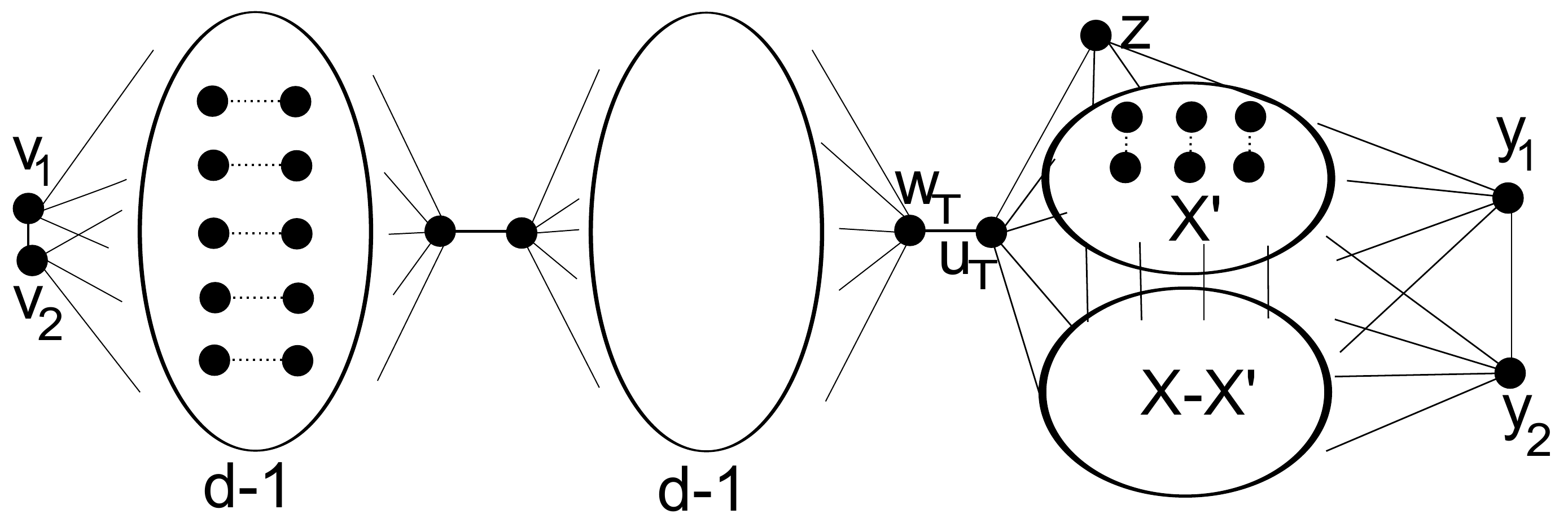}}
}
\caption{A tail}
\label{atail_figure}
\end{figure}

Either $H = H'$ or $H = H' - \{z\}$; this follows from Lemma \ref{generalregionlemma} and the observations that $N(x) = H'-\{z\}$ for $x \in X-X'$, and $N_2(x) = \{w_T,z\}$, whereas $w_T \not \in H$ since $w_T \in T$.  Let $R := T \cup H$.  We call $H$ the \textit{head region} of $R$.  Note that $G-\{z\}$ is disconnected with $R-\{z\}$ a component.  Also note that if we allow $|X'| = d-3$, then $z$ would have $d$ neighbors in $R$, and $G$ would be a snake graph.

Take $R \cap \mathcal{N}$ to be the two vertices $v_1, v_2 \in T$ that satisfy $\deg_2(v_1) = \deg(v_2) = 1$, together with $y_1$ and $y_2$.  We take $R \cap \mathcal{W} = \{u_T\}$.  Note that $z \in \mathcal{U}$.  Also, $X-X' \subset \mathcal{V}$, whereas $X'$ might or might not be a subset of $\mathcal{V}$.  Then $|R \cap \mathcal{V}| = k(d-1)+|X \cap \mathcal{V}|$.  Let $Z$ be the set of vertices adjacent to $z$ and not in $R$; $|Z| =d-|X'|-3 \geq 2$ by $|X'| < d-3$, $|X'|$ even, and $d$ odd.  Let $b_1$ and $b_2$ be two distinct vertices of $Z$.

Now we let $S_R$ be the following sets of pairs: 
\begin{list}{$\bullet$}{}
\item the $(4k-1)(d-1)$ elements of $S_{T1} \cup S_{T2} \cup S_{T3}$ (these are in $S_{R1}$);
\item the $2$ pairs $(u_T,y_1)$ and $(u_T,y_2)$ (these are in $S_{R1}$);
\item the $|X-X'|-1$ pairs $(u_T,s)$ for all $s \in Z$ (these are in $S_{R4}$);
\item the $|X'|$ pairs $(x,x')$ for $x, x' \in X'$ (these are in $S_{R1}$);
\item the $2|X-X'|$ pairs $(x,w_T)$, $(w_T,x)$ for $x \in X-X'$ (these are in $S_{R1}$);
\item the $|X-X'|$ pairs $(x,z)$ for $x \in X-X'$ (these are in $S_{R1}$ if $z \in R$ and otherwise in $S_{R2}$);
\item the $|X-X'|$ pairs $(z,x)$ for $x \in X-X'$ (these are in $S_{R1}$ if $z \in R$ and otherwise in $S_{R3}$);
\item the $2|X' \cap \mathcal{V}|$ pairs $(x,w_T)$, $(w_T,x)$ for $x \in X' \cap \mathcal{V}$ (if $|X \cap \mathcal{V}| > 0$, these are in $S_{R1}$).
\item the $2|X' \cap \mathcal{V}|$ pairs $(x,b_1)$, $(x,b_2)$ for $x \in X' \cap \mathcal{V}$ (if $|X \cap \mathcal{V}| > 0$, these are in $S_{R2}$).
\end{list}
Then $|S_R| = 4|\mathcal{V} \cap R|$.

\subsection{B tails}

Our final superregion is a \textit{B tail}, defined as follows.  See Figure \ref{btail_figure} for an illustration.

\begin{figure}[h]	
\centerline{
\mbox{\includegraphics[width=4.50in]{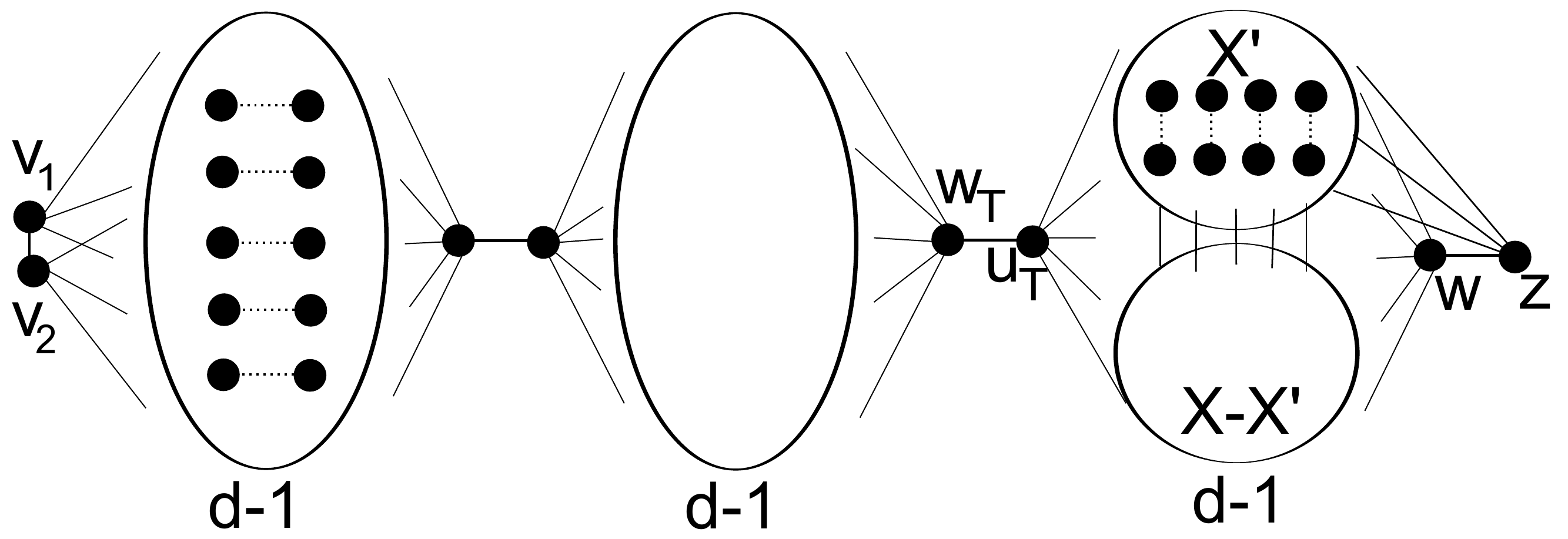}}
}
\caption{B tail}
\label{btail_figure}
\end{figure}

\begin{definition}
Let $T$ be a tail with $k$ segments, and let $H'$ be a subgraph of $G$ with the following vertices:
\begin{list}{$\bullet$}{}
\item $X$ with $|X| = d-1$; 
\item $X' \subset X$ such that $|X'|$ is even and not equal to $0$ or $d-1$; 
\item a vertex $w$; 
\item a vertex $z$; 
\item and $u_T$.
\end{list}
Suppose that $H'$ consists of the following edges: 
\begin{list}{$\bullet$}{}
\item a complete graph on $X$ with a matching on $X'$ removed;
\item all edges $u_Tx$ for $x \in X$;
\item all edges $wx$ for $x \in X$;
\item all edges $zx$ for $x \in X'$;
\item and $wz$.
\end{list}
Each $x \in X-X'$ is low degree and is contained in a common region $H$.  Then $T \cup H$ is a B tail.
\end{definition}

For $x \in X-X'$, $N_2(x) = N'_2(x) = \{z, w_T\}$.  Since $w_T \in T$, $w_T \not \in T$.  Also, since $d$ is odd, $|X'| \leq d-3$ and $z$ has $|X-X'| \geq 2$ neighbors outside of $R$.  Therefore, no vertex of $X'$ is low degree, and $w$ is low degree if and only if $|X'| = d-3$.  We conclude that $H = H'$ exactly when $|X'| = d-3$, and otherwise $H = H'-\{z\}$.  We say that $H$ is the \textit{head region} of $R$.  Note that $G-\{z\}$ is disconnected with $R-\{z\}$ a component.  Also note that if we allow $|X'| = d-1$, $R$ would be a snake graph, whereas if $|X'| = 0$, then $R$ would be an ordinary tail.

We take $R \cap \mathcal{W} = \{w\}$ and $R \cap \mathcal{N}$ to be the two vertices $v_1$ and $v_2$ of $T$ such that $\deg_2(v_1) = \deg_2(v_2) = 1$.  Also, $X-X' \subset \mathcal{V}$, whereas $X'$ might or might not be a subset of $\mathcal{V}$.  Note that $R \cap \mathcal{V} = k(d-1) + |X \cap \mathcal{V}|$.  Also, $z$ has $|X-X'|$ neighbors outside of $R$, and since $|X-X'| \geq 2$, consider distinct $b_1, b_2 \in N(z) - R$.  Then we define $S_R$ as the union of the following sets:
\begin{list}{$\bullet$}{}
\item the $(4k-1)(d-1)$ elements of $S_{T1} \cup S_{T2} \cup S_{T3}$ (these are in $S_{R1}$);
\item the $|X-X'|$ pairs $(w,s)$ for $s \in N(z) - \{w\}$ (these are in $S_{R4}$);
\item the $|X'|$ pairs $(x,x')$ for $x, x' \in X'$ and $xx'$ not an edge (these are in $S_{R1}$);
\item the $2|X|$ pairs $(x,w_T), (w_T,x)$ for $x \in X$ (these are in $S_{R1}$);
\item the $|X-X'|$ pairs $(x,z)$ for each $x \in X-X'$ (these are in $S_{R1}$ if $z \in H$ and otherwise in $S_{R2}$);
\item the $|X-X'|$ pairs $(z,x)$ for each $x \in X-X'$ (these are in $S_{R1}$ if $z \in H$ and otherwise in $S_{R3}$);
\item the $2|X' \cap \mathcal{V}|$ pairs $(x,b_1)$, $(x,b_2)$ for $x \in X' \cap \mathcal{V}$ (these are in$S_{R2}$).
\end{list}
Then $|S_R| \geq 4k(d-1)+4|X-X'|+4|X' \cap \mathcal{V}| = 4|\mathcal{V} \cap R|$.

\subsection{Identifying A and B tails}
In this section we prove an important lemma on the structure of A and B tails.

\begin{lemma}
\label{utlemma}
Let $H$ be a region of $G$, and suppose that $H$ is the head region of either an A or B tail.  Then exactly one vertex of $H$ is in $U := \{u_T: T$ is a tail$\}$.
\end{lemma}
\begin{proof}
Let $R$ be an A or B tail (possibly not unique) such that $H$ is the head of $R$.  Let $u_T$ and $z$ be as in the definition of $R$.  We show that $u_T$ is uniquely determined by $H$.  Define a \textit{link vertex} to be a vertex $a \in H$ such that $a$ is adjacent to exactly one vertex $v$ outside of $H$, and $v$ is adjacent to exactly one vertex in $H$, namely $a$.

First, suppose that $H$ contains a link vertex.  We show that $u_T$ is the only link vertex in $H$.  First suppose that $z \in H$.  Then $z$ and $u_T$ are the only vertices in $H$ that are adjacent to vertices outside of $H$.  But $z$ cannot be a link vertex, since $z$ has most $d-2$ neighbors in $H$ and thus at least $2$ outside of $H$.  Thus $u_T$ is the only link vertex.  Now, if $z \not \in H$, then the following vertices in $H$ are adjacent to vertices outside of $H$: $u_T$ and $N(z) \cap H$.  By construction, $|N(z) \cap H| \geq 2$, and thus no vertex in $N(z)$ is a link vertex and $u_T$ is the only link vertex.  Thus $u_T$ is uniquely determined when $H$ contains a link vertex.

Now suppose that $H$ does not contain a link vertex.  Then $H$ is the head region of an A tail and $z \not \in H$.  By construction, $u_T$ is the only vertex in $H$ that is adjacent to more than one vertex outside of $H$: namely $z$ and $w_T$.  Thus $u_T$ is uniquely determined in the case that $H$ does not contain a link vertex.
\end{proof}

\section{Superregions as a partition}
\label{part_sect}

In this section, we show that the superregions of $G$ partition $V(G)$ unless $G$ is a snake graph.

\begin{theorem}
Suppose that $G$ is not a snake graph.  Then $V(G)$ is partitioned by the superregions of $G$.
\end{theorem}

\begin{proof}
First, every vertex is contained in a superregion, since singleton sets are superregions if not contained in any larger superregion.  We need only to show that $R \cap R' = \emptyset$ if $R$ and $R'$ are distinct superregions.  If either $R$ or $R'$, say $R$, is a single vertex, then either $R \subset R'$ or $R \cap R' = \emptyset$.  The former is impossible by definition.  Now we assume that $R$ and $R'$ each consist of multiple vertices.

By construction, all superregions that consist of more than one vertex are unions of regions.  Therefore, if $R$ is both a region and a superregion, then either $R \subset R'$ or $R \cap R' = \emptyset$.  Again, the former is impossible by definition, and now we assume that $R$ and $R'$ each consist of multiple regions.  By construction, each of $R$ and $R'$ contain a tail.

Next, suppose that $R$ and $R'$ are both either a tail or a multitail, and suppose that $R \cap R' \neq \emptyset$.  Write $R = T_1 \cup \ldots \cup T_k$ and $R' = T'_1 \cup \ldots \cup T'_j$ as unions of tails.  By Lemma \ref{tail_intersect}, $R \cap R'$ is a collection of proper tails, and let $T$ be a proper tail in $R \cap R'$.  By $T \subset R$, for all proper tails $T' \in R$, $u_T = u_{T'}$.  Likewise, by $T \subset R'$, for all proper tails $T' \in R'$, $u_T = u_{T'}$.  Thus $R \cup R' = R = R'$.

Now let $R$ be an A or B tail with proper tail $T$, and let $R'$ be either a tail or multitail.  Write $R' = T'_1 \cup \cdots \cup T'_k$ as a union of disjoint tails.  The head region of an A or B tail is not isomorphic to any segment of a tail.  Thus, if $R \cap R' \neq \emptyset$, then $H \cap R' = \emptyset$ and $T \cap R' \neq \emptyset$.  By Lemma \ref{tail_intersect}, $T \subset R'$.  Let $H$ be the head region of $R$.  Then $u_T \in H$.  By construction of an A or B tail, $u_T$ has $d-1$ neighbors in $H$, and thus only $1$ neighbor in $R'$, which implies that $k=1$.  Then by Lemma \ref{tail_intersect}, $T = R'$ and $R'$ is not a superregion, a contradiction.

Finally, if $R$ and $R'$ are both A or B tails such that $R \cap R' \neq \emptyset$, then we show that $R=R'$.  Let $T$ and $T'$ be the respective tails of $R$ and $R'$, and let $H$ and $H'$ be the respective head regions.  By Lemma \ref{tail_intersect}, if $T \cap T' \neq \emptyset$, then $T=T'$.  Since only one region outside of $T$ may be adjacent to $T$, then $H = H'$ and thus $R = R'$.  Next, we have that $H \cap T' = H' \cap T = \emptyset$ by the fact that $H$ and $H'$ are not isomorphic to any segment of a tail.  Finally, suppose that $H = H'$.  By Lemma \ref{utlemma}, $u_T = u_{T'}$.  Neither $T$ nor $T'$ is part of a multitail since multitails have already been shown not to intersect A or B tails, and so $T = T'$.  This implies that $R=R'$.

This establishes that superregions partition $V(G)$.
\end{proof}

\section{Exceptions}
\label{exception_sect}
There are two families of graphs that are exceptions to Theorem \ref{main_theorem}.  In this section we discuss these exceptions in more detail.

\subsection{Snake graphs}
A snake graph $G$ consisting of $k \geq 2$ regions, as described above, has $n = k(d+1)+2$ vertices.  Since a snake graph contains a B region, a snake graph exists only if $d$ is odd.  A snake graph is determined, to isomorphism, by $d$ and $k$.

By construction, we may calculate that $$\sum_{v \in V(G)} \deg_2(v) = (4k-2)(d-1)+8 = \frac{(4k-2)(d-1)+8}{ k(d+1)+2}n.$$  For large $d$, this quantity is approximately $(4-2/k)n$.

\subsection{Peanut graphs}

A \textit{peanut graph} $G$ is defined as follows.  Partition $V(G)$ into sets $R_1$ and $R_2$ with $d+1$ and $d+2$ vertices respectively.  The only edge in the complement of $G[R_1]$ is $w_1w_2$.  The only edges in the complement of $G[R_2]$ are $uv_1, uv_2, uv_3$ and a matching on the remaining vertices.  The only edges between $R_1$ and $R_2$ are $uw_1$ and $uw_2$.  Note that $R_1$ and $R_2$ are both A regions.  A peanut graph exists only when $d$ is even, due to the matching in the complement of $R_2$, and is determined up to isomorphism by $d$.  See Figure \ref{peanut_figure} for an illustration.

\begin{figure}[h]	
\centerline{
\mbox{\includegraphics[width=3.00in]{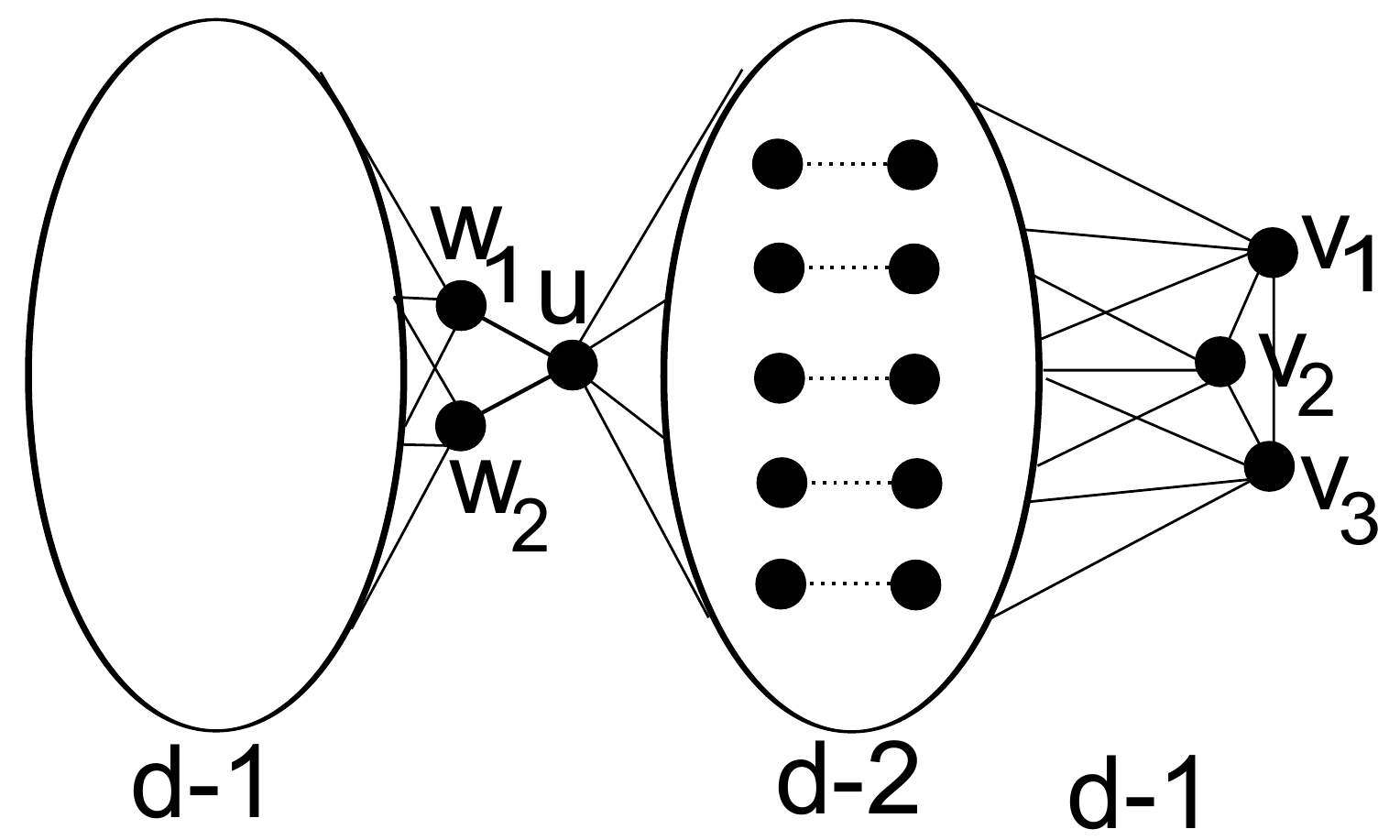}}
}
\caption{Peanut graph}
\label{peanut_figure}
\end{figure}

One may check that $n = 2d+3$ and $$\sum_{v \in V(G)} \deg_2(v) = 7d-4 = \frac{7d-4}{2d+3}n.$$  For large $d$, this quantity is approximately $(7/2)n$.

\section{Proof of Theorem \ref{main_theorem}}
\label{proof_sect}

In this section, we conclude the proof of Theorem \ref{main_theorem}.  We need to show that the $S_R$ are disjoint over all superregions $R$.  By Lemma \ref{duppairs}, we need only to consider $(x,y) \in S_{R4}$ for a superregion $R$ and show that, for each superregion $R'$, that $(x,y) \not \in S_{R'3}$, assuming that the elements of $S_{R'}$ are properly chosen.  Now suppose that $(x,y) \in S_{R'3}$.  Note that $x \in R \cap \mathcal{W}$, and $y$ is a low degree vertex of $R'$ with $x \in N_2'(y)$.  We consider several cases on $R$.

\subsection{$R$ is a single vertex, or a D, E, F, or G region}

This case is trivial, since by construction, $\mathcal{W} \cap R = \emptyset$ and $S_{R4} = \emptyset$.

\subsection{$R$ is an A region.}
\label{Aregionsubsection}

Let $X$ be the set of vertices $x'$ of $R$ that satisfy $\deg_2(x') = 2$.  Since $y$ is not in $R$, every vertex of $X$ is in $N_2(y)$.  Since the complement of the induced subgraph $G[X]$ is a matching, $|X|$ is even, and since $y$ is a low degree vertex, $|X| = 2$.  We must have $\deg_2(y) = 3$; otherwise, let $u$ be the one vertex adjacent to both $x$ and $y$, and let $Y := N(y) - \{u\}$ so that $|Y|=d-1$.  No vertex of $Y$ may have a neighbor outside of $Y \cup \{y,u\}$ if $\deg_2(y) = 2$, and thus $N(y') = Y -\{y'\} \cup \{u,y\}$ for $y' \in Y$.  But then $X \cup Y \subset N(u)$ is a contradiction to $\deg(u) = d$.

Let $z$ be the unique element of $N_2(y) - X$.  If $z \in N_2'(y)$, then in $S_{R'3}$, we may replace $(x,y)$ with $(z,y)$, unless, by Lemma \ref{duppairs}, $(z,y) \in S_{R''4}$ for some superregion $R''$.  Then $z \in \mathcal{W}$, and since $z$ is the only vertex in $R'' \cap N_2(y)$, $R''$ must be a tail with $z = w_{R''}$.  It must be that $u_{R''} \neq u$; otherwise, since $X \subset N(u)$, at most $d-2$ vertices of $Y = N(y) - \{u\}$ are adjacent to $u$, and so let $y' \in Y$ be not adjacent to $u$.  Then $N(y') \subset Y - \{y'\} \cup \{y\}$ by $\deg_2(y) = 3$, but then $y'$ has degree at most $d-1$, a contradiction.  Now, if $u \neq u'$, then $G[y, N(y)]$ has $d+1$ vertices, of which all have degree $d$, except that $u$ has degree $d-2$ and $u'$ has degree $d-1$.  This would imply an odd degree sum on $G[y, N(y)]$, which is impossible.

Now, consider the case that $z \in N_2(y)-N_2'(y)$.  Then $V(G) = R \cup Y \cup \{ u, y, z\}$.  Counting degrees on $G[Y \cup \{u,y,z\}]$, which has $d+2$ vertices, we see that $u$ has degree $d-2$ and all other vertices have degree $d$.  Let $a,b,c$ be the three vertices of $G[Y \cup \{u,y,z\}]$ that are not adjacent to $u$.  Then by degree considerations, the complement of $G[Y \cup \{u,y,z\}]$ contains edges $ua, ub, uc$ and a matching on all other vertices.  Thus $G$ is a peanut graph.

\subsection{$R$ is an $m$-multitail.}

First suppose that $m \geq 3$.  Then $y$ is distance $2$ from each vertex in $R \cap \mathcal{W}$, and since $y$ is a low degree vertex, $m=3$.  Let $u$ be the vertex adjacent to both $x$ and $y$, and let $Y := N(y) - \{u\}$.  Since $u$ is adjacent to $y$ and $3$ vertices in $R$, $u$ has at most $d-4$ neighbors in $Y$.  Choose $y' \in Y - N(u)$.  Since $\deg_2(y) = 3$, $N(y') \subset Y - \{y'\} \cup \{y\}$.  Then $y'$ has degree at most $d-1$, a contradiction.  We conclude that $m=2$.

With $m=2$, we may perform the same analysis as in Section \ref{Aregionsubsection} and conclude that $R'$ is the complement of a graph with $d+2$ vertices containing edges $ua, ub, uc$ and a matching on all other vertices.  This implies that $d$ is even, and the existence of a tail implies that $d$ is odd, and so $R$ cannot be a multitail.

\subsection{$R$ is a C region.}

Let $u$ be the unique vertex that is adjacent to both $x$ and $y$.  By Lemma \ref{cregionlemma}, $R$ contains a vertex $x' \neq x$ in $N(u)$, and $\{x,x'\} \subseteq N_2(y)$.  Suppose, by way of contradiction, that $\deg_2(y) = 2$.  Let $Y := N(y) - \{u\}$.  If $y' \in Y$, then $N(y') \subset Y - \{y'\} \cup \{u,y\}$, and since $\deg_2(y') = d$, then $N(y') = Y - \{y'\} \cup \{u,y\}$.  Then $Y \cup \{x,x'\} \subset N(u)$, a contradiction to $\deg_2(u) = d$.  We conclude that $\deg_2(y) = 3$.

Note that $x' \not \in \mathcal{W}$.  By Lemma \ref{lemma31} and the fact that $\deg_2(y) = 3$, if $(\tilde{x},y) \in S_{R
3}$, then $\tilde{x} = x$.  Then in $S_{R'3}$, we may replace $(x,y)$ by $(x',y)$, and since $x' \not \in \mathcal{W}$, $(x,y) \not \in S_{R''4}$ for any superregion $R''$.

\subsection{$R$ is an A tail or a B tail.}

Let $u$ be the unique vertex that is adjacent to both $x$ and $y$.  By construction, there are at least three vertices $x,x', x''$ in $R \cap N_2(y)$.  Let $Y := N(y) - \{u\}$.  Since $y$ is low degree, for each $y' \in Y$, $N(y') \subset Y - \{y'\} \cup \{u,y\}$.  Since $\deg(y') = d$, $N(y') = Y - \{y'\} \cup \{u,y\}$.  Then $Y \cup \{x,x',x''\} \subset N(u)$, a contradiction to $\deg_2(u) = d$.

\subsection{$R$ is a tail.}

For the final case, that $R$ is a tail, we consider several cases on $R'$.  Let $u$ be the one vertex that is adjacent to both $x$ and $y$.

\subsubsection{$R'$ is an A or B region}

If $R'$ is an A region, then let $Y := N(u) \cap R'$.  By definition of an A region, $|Y| < d-1$.  Also $|Y|$ must be even, and $d$ must be odd by the existence of a tail $R$, and so $|Y| \leq d-3$.  Then $N(u)$ consists of at least $3$ vertices outside of $R'$.  Also, $N(u)-R' \subset N_2(y)$, and $N_2(y)$ consists of at least one vertex in $R'$, contradicting the definition of a low degree vertex.  So $R'$ is not an A region.  If $R'$ is a $B$ region, then $G$ is a snake graph.

\subsubsection{$\deg_2(y) = 2$ and $|N_2'(y)| = 1$}

Let $y'$ be the one vertex in $N_2(y) - N_2'(y)$.  Since $u$ has $d-2$ neighbors outside of $R$, excluding $y$, and $y$ has $d-1$ neighbors excluding $u$, $y$ has a neighbor $z$ that is not adjacent to $u$.  Then $N(z) \subset N(y) \cup \{y,y'\} - \{z,u\}$, and since $\deg(z) = d$, $N(z) = N(y) \cup \{y,y'\} - \{z,u\}$.  The only vertices in the component of $G-\{u\}$ that contains $y$ are $N(y) - \{u\} \cup \{y,y'\}$, and thus $N_2(z) = \{u\}$, which implies that $R'$ is an A or B region.  This case has already been addressed.

\subsubsection{$\deg_2(y) = 2$ and $|N_2'(y)| = 2$}

Let $N_2'(y) = \{x,z\}$.  We consider two cases: if $z$ and $u$ are neighbors, and if they are not neighbors.

If $z$ and $u$ are neighbors, then $u$ has a set $Y$ of $d-3$ neighbors outside of $\{x,y,z\}$, and $y$ has $d-1$ neighbors, excluding $u$.  Let $w$ and $w'$ be distinct vertices in $N(y) - N(u) - \{u\}$.  Every vertex in $Y$ is within distance $2$ of $y$, and since $N_2(y) \cap Y = \emptyset$, $Y \subset N(y)$.  Thus $Y \subset R'$.  Also, $w$ and $w'$ are adjacent to each vertex in $Y \cup \{y,z\}$ as well as each other, since neither are adjacent to $u$, $x$, or any vertex of distance $3$ or more from $y$.  Similarly, every vertex in $Y$ can only have neighbors among $Y \cup \{y, w, w', z, u\}$, a set of size $d+2$, and so $y' \in Y$ is adjacent to all but possibly one other vertex in $Y$.  The complement of $G[Y]$ has no edges except a (possibly empty) matching, and let $Y'$ be the set of such vertices that are in such a matching.  Every vertex in $Y'$ is adjacent to $z$.  If $y' \in Y-Y'$, then $y'$ is not adjacent to $z$ since $y'$ has been established to be adjacent to every other vertex in $Y \cup \{y, w, w', z, u\}$ besides itself.  We conclude that $R \cup R'$ is an A tail, a contradiction to the assumption that $R$ is a superregion.

Now, if $u$ and $z$ are not neighbors, let $Y := N(y) \cap N(u)$.  Since $\{x,y\} \subset N(u) - N(y)$, $|Y| \leq d-2$.  Since $N(u) - N(y) -\{y\} \subset N_2(y)$, in fact $N(u) - N(y) = \{x,y\}$ and $|Y| = d-2$.  Choose $u'$ so that $N_2(y) = Y \cup \{u,u'\}$.  By $\deg_2(y) = 2$, $u'$ can have no neighbors outside of $Y \cup \{y,z\}$, and so $N(u') = Y \cup \{y,z\}$.  For $y' \in Y$, the only possible neighbors of $y'$ are in $Y \cup \{u, y, u', z\}$, since $y'$ is not adjacent to $x$ or any vertex of distance $3$ or more from $y$.  Of the vertices of $Y \cup \{u, y, u', z\} - \{y'\}$, $y'$ is adjacent to all but $1$.  Let $Y' := Y \cap N(z)$.  Each $y' \in Y'$ is adjacent to $u,u',y,z$ and thus all but $1$ other vertex of $Y$.  If $y' \in Y - Y'$, then $y$ is adjacent to every vertex of $Y - \{y'\}$, and so the complement of $G[Y]$ is a matching on $Y'$.  If $Y' \neq \emptyset$, then $R \cup R'$ is a B tail, a contradiction to the assumption that $R$ is a superregion, and if $Y = \emptyset$, then $R \cup R'$ is a tail, also a contradiction.

\subsubsection{$\deg_2(y) = 3$ and $|N_2'(y)| = 1$}

Let $G_y$ be the component of $G-x$ that contains $y$.  Then $G_y$ has $d+3$ vertices, namely $y$, all neighbors of $y$, and $N_2(y) - \{x\}$.  Each vertex in $G_y$ has degree $d$, except that $u$ has degree $d-1$.  This is impossible since the degree sum would be odd.

\subsubsection{$\deg_2(y) = 3$ and $|N_2'(y)| = 2$}

Let $N_2'(y) = \{x,z\}$.  By Lemma \ref{lemma31}, in $S_{R'3}$, we may replace $(x,y)$ with $(z,y)$, unless $(z,x) \in S_{R''4}$ for some superregion $R''$.  In this case, since $z \in \mathcal{W}$, and $R''$ has no vertex besides $z$ in $N_2(y)$, $R''$ is a tail.  Let $u'$ be the unique that is adjacent to both $y$ and $z$.  Assume that $u \neq u'$, since otherwise $R$ is a multitail.  Let $G_y$ be the component of $G-x-z$ that contains $y$.  Then $G_y$ has $d+2$ vertices, namely $y$, $N(y)$, and the one vertex of $N_2(y) - N_2'(y)$.  In $G_y$, all vertices have degree $d$ except for $u$ and $u'$, which each have degree $d-1$.  This requires $d$ to be even, so that sum of the degrees of all vertices in $G_y$ is even.  However, the existence of a tail $R$ requires $d$ to be odd.  We conclude that $(z,x) \not \in S_{R''4}$ as desired.

\subsubsection{$\deg_2(y) = 3$ and $|N_2'(y)| = 3$}

Let $N_2(y) = N_2'(y) = \{x,z,z'\}$.  By Lemma \ref{lemma31}, if $z \not \in \mathcal{W}$, then in $S_{R'3}$, we may replace $(x,y)$ with $(z,y)$.  Likewise, if $z' \not \in \mathcal{W}$, then in $S_{R'3}$, we may replace $(x,y)$ with $(z',y)$.  Now suppose that both $z$ and $z'$ are in $\mathcal{W}$.  If $z$ and $z'$ are in the same region $R''$, then $R''$ is either an A region or a multitail.  In that case, let $u_1$ be the unique vertex adjacent to each of $y, z, z'$.  Then in $G[y, N(y)]$, all vertices have degree $d$ except for $u$ and $u'$.  If $u = u'$, then $\deg(u) = d-3$ in $G[y, N(y)]$, while if $u \neq u'$, then $u$ and $u'$ have degree $d-1$ and $d-2$ respectively in $G[y, N(y)]$.  Both of these cases are impossible since the degree sum would be odd.

Now suppose that $x, z$, and $z'$ are all in different regions.  Since the regions containing $x, z, z'$ respectively each have exactly one vertex in $N_2(y)$, they must all be tails.  Let $u, u_1, u_2$ be the vertices adjacent to $y$ and respectively $x,z,z'$.  Since each of $x,z,z'$ are contained in tails and not multitails, $u,u_1,u_2$ are distinct.  Then $G_y$, the induced subgraph consisting of $y$ and its neighbors, has $d+1$ vertices, and all vertices have degree $d$ except for $u, u_1, u_2$, which each have degree $d-1$.  This is also impossible, since the sum of the degrees would be odd.

This enumerates all cases.

\end{document}